\documentclass{article}
\usepackage{amsmath, amsfonts, graphicx, amsthm,color,enumerate}

 \usepackage{tikz}

\usepackage{cite}

\numberwithin{equation}{section}

\DeclareMathOperator{\Inn}{Inn}
\DeclareMathOperator{\Aut}{Aut}
\DeclareMathOperator{\diag}{diag}
\DeclareMathOperator{\id}{id}
\DeclareMathOperator{\ch}{char}
\DeclareMathOperator{\GL }{GL }
\DeclareMathOperator{\SL}{SL}
\DeclareMathOperator{\Sp}{Sp}
\DeclareMathOperator{\SO}{SO}

\DeclareMathOperator{\PGL}{PGL}

\newcommand{\I}{\mathcal{I}}

\title{Isomorphy classes of $k$-involutions of $G_2$}
\author{John Hutchens  \\ \emph{jdhutche@ncsu.edu} \\ \\ \emph{Department of Mathematics} \\ \emph{North Carolina State University} \\ \emph{Raleigh, NC  27695} }

\begin{document}

\maketitle

\begin{abstract}
Isomorphy classes of $k$-involutions have been studied for their correspondence with symmetric $k$-varieties, also called generalized symmetric spaces.  A symmetric $k$-variety of a $k$-group $G$ is defined as $G_k/H_k$ where $\theta:G \to G$ is an automorphism of order $2$ that is defined over $k$ and $G_k$ and $H_k$ are the $k$-rational points  of $G$ and $H=G^{\theta}$, the fixed point group of $\theta$, respectively.  This is a continuation of papers written by A.G. Helminck and collaborators \cite{He00}, \cite{DHW06}, \cite{DHW-}, \cite{HW02} expanding on his combinatorial classification over certain fields.  Results have been achieved for groups of type $A$, $B$ and $D$.  Here we begin a series of papers doing the same for algebraic groups of exceptional type.
\end{abstract}

\section{Introduction}

The problem of identifying all isomorphy classes of symmetric $k$-varieties is described by Helminck in \cite{He94}.  There he notes that isomorphy classes of symmetric $k$-varieties of algebraic groups and isomorphy classes of their $k$-involutions are in bijection.  In the following we provide a classification of isomorphy classes of $k$-involutions for the split type of $G_2$ over certain fields. \

The main result of this paper is an explicit classification of $k$-involutions of the split form of $G_2$ where $k=\mathbb{R},\mathbb{C},\mathbb{Q}, \mathbb{Q}_p,$ and $\mathbb{F}_q$, where $q>2$.  We do this by finding explicit elements of $\Aut(G_2)$, where $G_2=\Aut(C)$ and $C$ is always the split octonion algebra over a given field of characteristic not $2$. \

The results from this paper rely most heavily on the works of Jacobson \cite{Ja58} on composition algebras, Lam's presentation of quadratic forms \cite{La05}, and Helminck et. al. on symmetric spaces and $k$-involutions of algebraic groups.

A \emph{$k$-involution} is an automorphism of order exactly 2, that is defined over the field $k$.  The isomorphy classes of these $k$-involutions are in bijection with the quotient spaces $G_k/H_k$, where $G_k$ and $H_k$ are the $k$-rational points of the groups $G$ and $H=G^{\theta}=\{g\in G \ | \ \theta(g)=g \}$ respectively, these quotient spaces will be called \emph{symmetric $k$-varieties} or \emph{generalized symmetric spaces}. \

The group of characters and root space associated with a torus $T$ are denoted by $X^*(T)$ and $\Phi(T)$ respectively.  We will also denote by 
	\[ A_{\theta}^- = \{ a \in A \ | \ \theta(a) = a^{-1} \}^{\circ}, \]
and by
	\[ I_k(A_{\theta}^-) = \{ a \in A_{\theta}^- \ | \ \theta \circ \Inn(a) \text{ is a $k$-involution} \}. \]

In the following we introduce a  characterization of $k$-involutions of an algebraic group of a specific type given by Helminck.  The full classification can be completed with the classification of the following three types of invariants, \cite{He00},

\begin{enumerate}[(1)]
\item classification of admissible involutions of $(X^*(T),X^*(A), \Phi(T), \Phi(A))$, where $T$ is a maximal torus in $G$, A is a maximal $k$-split torus contained in $T$
\item classification of the $G_k$-isomorphy classes of $k$-involutions of the $k$-anisotropic kernel of $G$
\item classification of the $G_k$-isomorphy classes of $k$-inner elements $a\in I_k(A_{\theta}^-)$.
\end{enumerate}

For this paper we do not consider $(2)$ since our algebraic groups will be $k$-split.  We mostly focus on (3) and refer to (1) when appropriate, though Helminck has provided us with a full classification of these \cite{He00}.  \

The main result is an explicit description of the $k$-inner elements up to isomorphy, which completes the classification of $k$-involutions for the split group of type $G_2$.  Each admissible involution of $(X^*(T), X^*(A), \Phi(T), \Phi(A) )$ can be lifted to a $k$-involution $\theta$ of the algebraic group.  This lifting is not unique.  The involutions that induce the same involution as $\theta$ when restricted to \\ $(X^*(T), X^*(A), \Phi(T), \Phi(A) )$ are of the form $\theta \circ \Inn(a)$ where $a \in A_{\theta}^-$.  This set of elements such that $\theta \circ \Inn(a)$ is a $k$-involution form the set of \emph{$k$-inner elements} associated with the involution $\theta$ and are denoted $I_k(A_{\theta}^-)$.   \

Yokota wrote about $k$-involutions, $\theta$, and fixed point groups, $G^{\theta}$, for algebraic groups of type $G_2$ for $k=\mathbb{R}, \mathbb{C}$. In fact the elements of $G_2=\Aut(C)$ we call $\mathcal{I}_{t_{(\pm 1, \pm 1)}}$ correspond to the $\gamma$ maps in \cite{Yo90}, which are a conjugation with respect to complexification at different levels within the octonion algebra taken over $\mathbb{R}$ or $\mathbb{C}$  \cite{Yo90} .  We found these using different methods than in \cite{Yo90}, and show the correspondence.  \

In section 3.1 we find $k$-involutions that come from conjugation by elements in a maximal $k$-split torus and using a result of Jacobson show they are isomorphic in Proposition \ref{tconj}.  This will take care of all cases where $\Aut(C)= G_2$ is taken over a field whose structure permits a split octonion algebra and only split quaternion algebras, namely $k=\mathbb{C}$ and $\mathbb{F}_p$ when $p>2$.  \

Over other fields there is the possibility of division quaternion algebras, and this fact using Proposition \ref{jake} gives us another isomorphy class of $k$-involutions when we take $k=\mathbb{R}, \mathbb{Q}$ and $\mathbb{Q}_p$.  In section 3.2 we find the $k$-involution $\theta$ for which our maximal $k$-split torus is a maximal $(\theta,k)$-split torus and in Lemma \ref{nonsplitq} we find a representative for the only other possible isomorphy class of $k$-involutions over these fields. \

In section 3.3 we summarize the main results of the paper in Theorem \ref{maintheorem} and give the full classification of $k$-involutions when $k=\mathbb{C}, \mathbb{R}, \mathbb{Q}_p$ and $\mathbb{F}_q$ when $p \geq 2$ and $q>2$.  We finish up by giving descriptions of the fixed point groups of isomorphy classes of $k$-involutions.

\subsection{Preliminaries and recollections}

Most of our notation is borrowed from \cite{Sp98} for algebraic groups, \cite{He00} for $k$-involutions and generalized symmetric spaces, \cite{La05} for quadratic forms and \cite{SV00} for octonion and quaternion algebras. \

The letter $G$ is reserved for an arbitrary reductive algebraic group unless it is $G_2$, which is specifically the automorphism group of an octonion algebra.  When we refer to a maximal torus we use $T$ and any subtorus is denoted by another capital letter, usually $A$.  Lowercase Greek letters are field elements and other lowercase letters denote vectors.  We use $Z(G)$ to denote the center of $G$, $Z_G(T)$ to denote the centralizer of $T$ in $G$, and $N_G(T)$ to denote the normalizer of $T$ in $G$.  \

By $\Aut(G)$ we mean the automorphism group of $G$, and by $\Aut(C)$ we mean the linear automorphisms of the composition algebra $C$.  The group of inner automorphisms are denoted $\Inn(G)$ and the elements of $\Inn(G)$ are denoted by $\mathcal{I}_g$ where $g\in G$ and $\mathcal{I}_g(x) = gxg^{-1}$. \

We define a \emph{$\theta$-split torus}, $A$, of an involution, $\theta$, as a torus $A  \subset G$ such that $\theta(a)=a^{-1}$ for all $a\in A$.  We call a torus \emph{$(\theta, k)$-split} if it is both $\theta$-split and $k$-split.  \

Let $A$ be a $\theta$-stable maximal $k$-split torus such that $A_{\theta}^-$ is a maximal $(\theta,k)$-split torus.  By \cite{HW02} there exists a maximal $k$-torus $T \supset A$ such that $T_{\theta}^- \supset A_{\theta}^-$ is a maximal $\theta$-split torus.  The involution $\theta$ induces an involution $\theta \in \Aut(X^*(T), X^*(A), \Phi(T), \Phi(A))$.  It was shown by Helminck \cite{He00} that such an involution is unique up to isomorphy.  For $T\supset A$ a maximal $k$-torus, an 
\[ \theta \in \Aut(X^*(T), X^*(A), \Phi(T), \Phi(A)) \] is \emph{admissible} if there exists an involution $\tilde{\theta} \in \Aut(G,T,A)$ such that $\tilde{\theta}|_T = \theta$, $A_{\theta}^-$ is a maximal $(\theta, k)$-split torus, and $T_{\theta}^-$ is a maximal $\tilde{\theta}$-split torus of $G$.  \

This will give us the set of $k$-involutions on $G$ that extend from involutions on the group of characters, $X^*(T)$.\

As for the $k$-inner elements they are defined as follows;  if $\theta$ is a $k$-involution and $A_{\theta}^-$ is a maximal $\theta$-split torus then the elements of the set,
\[ I_k(A_{\theta}^-) = \left\{ a \in A_{\theta}^- \ \big| \ \left(\theta \circ \mathcal{I}_a \right)^2 = \id, \ \left(\theta \circ \mathcal{I}_a\right)(G_k)=G_k  \right\}, \]
are called \emph{$k$-inner elements} of $\theta$.  Some compositions $\theta \circ \mathcal{I}_a$ will not be isomorphic in the group $\Aut(G)$ for different $a\in I_k(A_{\theta}^-)$, though they will project down to the same involution of the group of characters of a maximal torus fixing the characters associated with a maximal $k$-split subtorus for all $a\in I_k(A_{\theta}^-)$.  \

\subsection{Split octonion algebra}

Throughout this paper we use $N$ to be a quadratic form of a composition algebra and $\langle \ , \ \rangle$ to be the bilinear form associated with $N$.  The capital letters $C$ and $D$ denote composition algebras and composition subalgebras respectively.  The composition algebras we will refer to always have an identity, $e$.  There is an anti-automorphism on a composition algebra that resembles complex conjugation denoted by, $\bar{ \ }$, which will have specific but analogous definitions depending on the dimension of the algebra.  \

If $\{ e, a, b, ab\}$ is a basis for $D$, a quaternion algebra such that $a^2 = \alpha$ and $b^2 = \beta$, then we denote its quadratic from by, $\left( \frac{\alpha,\beta}{k} \right) \cong \langle 1, -\alpha, -\beta, \alpha\beta \rangle$.  We will often refer to a quaternion algebra over a field $k$ by the $2$-Pfister form notation of its quadratic form.  Since a composition algebra is completely determined by its quadratic form and its center, $ke$, there is no risk of ambiguity.  By $k$ we are referring to an arbitrary field and by using the blackboard bold $\mathbb{F}$ we refer to a specific field.  We consider only fields that do not have characteristic $2$.  \

We always consider the octonion algebra as a doubling of the split quaternions thought of as $M_2(k)$, the $2 \times 2$ matrices over our given field.  The octonion algebras we consider are an ordered pair of these with an extended multiplication that will be described in the next section.

For the following results we refer you to \cite{SV00} for the proofs.  In these upcoming results $V$ is a vector space over a field, $k$, $\ch(k)\neq 2$, equipped with a quadratic form $N:V \to k$  and the associated bilinear form $\langle \ , \ \rangle:V\times V \to k$, such that
\begin{enumerate}
\item $N(\alpha v) = \alpha^2 N(v)$,
\item $\langle v,w \rangle = N(v+w) - N(v) - N(w)$,
\end{enumerate}
where, $v,w \in V$ and $\alpha \in k$.  

A \emph{composition algebra}, $C$, will be a vector space with identity, $e$, and $N$ and $\langle \ , \ \rangle$ as above such that $N(xy)=N(x)N(y)$.  Note that $ke$ is a composition algebra in a trivial way.

\newtheorem{doubling}[subsubsection]{Proposition, \cite{SV00}}
\begin{doubling}
When $C$ is a composition algebra with $D\subset C$ a finite dimensional subalgebra of $C$ with $C\neq D$, then we can choose $a \in D^{\perp}$ with $N(a) \neq 0$, then $D \oplus Da$ is a composition algebra.  The product, quadratic form, and complex conjugation are given by
\begin{enumerate}
\item $(x+ya)(u+va) = (xu + \alpha \bar{v}y) + (vx + y\bar{u})a$, \text{ for } $x,y,u,v \in D$, $\alpha \in k^*$
\item $N(x+ya) = N(x) - \alpha N(y)$
\item $\overline{x+ya} = \bar{x} - ya$.
\end{enumerate}
The dimension of $D\oplus Da$ is twice the dimension of $D$ and $\alpha = -N(a)$.
\end{doubling}

We often use this decomposition above in the results that follow, and a theorem of Adolf Hurwitz gives us the possible dimensions of such algebras.

\newtheorem{hurwitz}[subsubsection]{Theorem, (A. Hurwitz), \cite{SV00}}
\begin{hurwitz}
\label{hurwitz}
Every composition algebra can be obtained from iterations of the doubling process starting from $ke$.  The possible dimensions of a composition algebra are $1,2,4$, or $8$.  A composition algebra of dimension $1$ or $2$ is commutative and associative, a composition algebra of dimension $4$ is associative and not commutative, a composition algebra of dimension $8$ is neither commutative nor associative.
\end{hurwitz}

\newtheorem{doublesplit}[subsubsection]{Corollary, \cite{SV00}}
\begin{doublesplit}
Any doubling of a split composition algebra is again a split composition algebra.
\end{doublesplit}

There are $2$ general types of composition algebras.  If there are no zero divisors we call the composition algebra a \emph{division algebra}, and otherwise we call it a \emph{split algebra}.  It follows from the definition that a composition algebra is determined completely by its norm, and we have the following theorem.  

\newtheorem{splitcomp}[subsubsection]{Theorem, \cite{SV00}}
\begin{splitcomp}
\label{splitcomp}
In dimensions $2,4,$ and $8$ there is exactly one split composition algebra, over a given field $k$, up to isomorphism.
\end{splitcomp}

We can take as a split quaternion algebra, $D$, over a field $k$ to be $M_2(k)$, the $2\times 2$ matrices over $k$.  Multiplication in $D$ will be typical matrix multiplication, our quadratic form will be given by
\[ N\left(
\begin{bmatrix}
x_{11} & x_{12} \\
x_{21} & x_{22}
\end{bmatrix}
\right)
=
\text{det} \left(
\begin{bmatrix}
x_{11} & x_{12} \\
x_{21} & x_{22}
\end{bmatrix}
\right)
= x_{11} x_{22} - x_{12}x_{21},
\] 
and \emph{bar involution} will be given by
\[
\overline{
\begin{bmatrix}
x_{11} & x_{12} \\
x_{21} & x_{22}
\end{bmatrix}
}
=
\begin{bmatrix}
x_{22} & -x_{12} \\
-x_{21} & x_{11}
\end{bmatrix}.
\]
Elements of our split octonion algebra will have the form
\[ (x,y) =\left(
\begin{bmatrix}
x_{11} & x_{12} \\
x_{21} & x_{22}
\end{bmatrix},
\begin{bmatrix}
y_{11} & y_{12} \\
y_{21} & y_{22}
\end{bmatrix}
\right).
\]
Since all split octonions over a given field are isomorphic, we can take $\alpha =1$ in our composition algebra doubling process.  The multiplication, quadratic form, and octonion conjugation are given by the following;
\begin{align}
(x,y)(u,v) &= (xu + \bar{v}y, vx + y \bar{u}), \\
N\big( (x,y) \big) &= \text{det}(x) - \text{det}(y), \\
\overline{(x,y)} &= (\bar{x},-y),
\end{align}
with $x,y,u,v \in M_2(k)=D$.  The basis of the underlying vector space is taken to be $\left\{ (E_{ij}, 0) , ( 0, E_{ij}) \right\}_{i,j=1,2}$, where the $E_{ij}$ are the standard basis elements for $2 \times 2$ matrices, and so our identity element in $C$ is $e = (E_{11} + E_{22},0)$.

\section{Automorphisms of $G_2$}

\subsection{Some results on $G_2$}

It is well known that the automorphism group of a $k$-split octonion algebra, $C$, over a field, $k$, is a $k$-split linear algebraic group of type $G_2$ over $k$.  We can compute a split maximal torus for $\Aut(C)$, where $C= \big( M_2(k), M_2(k) \big)$ as above.  Here we collect some some known results, and again we refer to \cite{SV00}.

\newtheorem{maximaltorus}[subsubsection]{Theorem}
\begin{maximaltorus}
The following statements concerning $G= \Aut(C)$ are true;
\begin{enumerate}
\item There is a maximal $k$-split torus $T \subset G$ of the form
\[ T = \left\{ \diag(1,\beta \gamma, \beta^{-1} \gamma^{-1}, 1, \gamma^{-1}, \beta, \beta^{-1}, \gamma ) \ \big| \beta, \gamma \in k^* \right\}.  \]
\item The center of $\Aut(C)$ contains only the identity.
\item For any composition algebra, $C$, over $k$.  The only nontrivial subspaces of $C$ left invariant by $\Aut(C)$ are $ke$ and $e^{\perp}$.
\item All automorphisms of $\Aut(C)$ are inner automorphisms.
\end{enumerate}
\end{maximaltorus}

\subsection{$k$-involutions of $G$}

\newtheorem{kinv}[subsubsection]{Remark}
\begin{kinv}
If $\theta \in \Inn(G)$ and $\theta =  \mathcal{I}_t$ is a $k$-involution then $t^2 \in Z(G)$.
\end{kinv}

Since groups of type $G_2$ have a trivial center, the problem of classifying $k$-involutions for $\Aut(C)$, where $C$ is a split octonion algebra, is the same as classifying the conjugacy classes of elements of order $2$ in $\Aut(C)$ that preserve the $k$-structure of $\Aut(C)$. \

\newtheorem{torusinv}[subsubsection]{Remark}
\begin{torusinv}
\label{torusinv}
The involutions that are of the form $\mathcal{I}_t$ where 
\[ t_{(\beta, \gamma)} \in T = \left\{ \diag(1,\beta \gamma, \beta^{-1} \gamma^{-1}, 1, \gamma^{-1}, \beta, \beta^{-1}, \gamma ) \ \big| \beta, \gamma \in k^* \right\} \]
have $(\beta, \gamma) = (1,-1),(-1,1)$ or $(-1,-1)$.
\end{torusinv}
\begin{proof}
We set $\diag(1,\beta \gamma, \beta^{-1} \gamma^{-1}, 1, \gamma^{-1}, \beta, \beta^{-1}, \gamma )^2 = \id$ and exclude $(\beta, \gamma)=(1,1)$ which corresponds to the identity map.
\end{proof}

Using the above statement and the following result of Jacobson we can show that all $k$-involutions given by conjugation by elements coming from the maximal $k$-split torus $T$ are isomorphic.

\newtheorem{jake}[subsubsection]{Proposition, \cite{Ja58}}
\begin{jake}
\label{jake}
Let $C$ be an octonion algebra over $k$, then the conjugacy class of quadratic elements, $t\in G=\Aut(C)$ such that $t^2=\id$ are in bijection with the isomorphism classes of quaternion subalgebras of $C$.
\end{jake}

In particular if $t \in \Aut(C)$ has order $2$, then it leaves some quaternion subalgebra $D$ elementwise fixed giving us the eigenspace corresponding to the eigenvalue $1$.  Then $D^{\perp}$ is the eigenspace corresponding to the eigenvalue $-1$.  If $gtg^{-1}=s$ for some $g\in G$, then $s$ has order $2$ and $g(D)=D'$, $D'$ a quaternion subalgebra  elementwise fixed by $s$, and $D\cong D'$.

\newtheorem{qsubalg}[subsubsection]{Corollary}
\begin{qsubalg}
Let $C$ be an octonion algebra over $k$ and $D$ and $D'$ quaternion subalgebras of $C$.  If $s,t \in G = \Aut(C)$ are elements of order $2$ and $s,t$ fix $D,D'$ elementwise respectively, then $s \cong t$ if and only if $D \cong D'$ over $k$.
\end{qsubalg}

\newtheorem{conjint2}[subsubsection]{Corollary}
\begin{conjint2}
\label{conjint2}
For $s,t \in \Aut(C)$, $\mathcal{I}_t \cong \mathcal{I}_s$ if and only if $s$ and $t$ leave isomorphic quaternion subalgebras invariant.
\end{conjint2}

We can begin looking for elements of order $2$ in the $k$-split maximal torus we have computed, which in our case will be $t \in T$, $t^2 = \id$.  Solving the following equation
\[ \diag(1,\beta \gamma, \beta^{-1} \gamma^{-1}, 1, \gamma^{-1}, \beta, \beta^{-1}, \gamma )^2 = \id, \]
we obtain the elements
\[ t_{(\beta,\gamma)} \text{ where } (\beta ,\gamma) = (\pm 1,\pm 1), \text{ and } (\beta,\gamma) \neq (1,1). \]

\newtheorem{normalconjlemma}[subsubsection]{Lemma}
\begin{normalconjlemma}
\label{normalconjlemma}
$\I_g \I_{\varepsilon}\I_g^{-1} = \I_{g\epsilon g^{-1}}$
\end{normalconjlemma}
\begin{proof}
We can apply the left hand side to an element $y\in G$,
\begin{align*}
\I_g\I_{\varepsilon}\I_g^{-1} (y) &= g \big( \varepsilon (g^{-1} y g) \varepsilon^{-1} \big) g^{-1} \\
&= (g \varepsilon g^{-1}) y (g \varepsilon g^{-1})^{-1} \\
&= \I_{(g \varepsilon g^{-1})}(y).
\end{align*}
\end{proof}

\newtheorem{normalconj}[subsubsection]{Proposition}
\begin{normalconj}
If $\varepsilon_1^2=\varepsilon_2^2 = \id$ and $\varepsilon_1, \varepsilon_2  \in T$ a maximal torus of $G$ when $Z(G)=\{ \id \}$, then $\I_{\varepsilon_1} \cong \I_{\varepsilon_2}$ if and only if $\varepsilon_1 = n \varepsilon_2 n^{-1}$ for some $n \in N_G(T)$.
\end{normalconj}
\begin{proof}
If $n\varepsilon_2 n^{-1} = \varepsilon_1$ for $n\in N_G(T)$ then $\I_{\varepsilon_1} \cong \I_{\varepsilon_2}$ via the isomorphism $\Inn \left( \I_n \right)$. 

Now we let $\I_{\varepsilon_1} \cong \I_{\varepsilon_2}$, and so there exists a $g\in G$ such that $\I_{\varepsilon_1}=\I_g \big(\I_{\varepsilon_2} \big) \I_g^{-1}$, 
then we have by Lemma \ref{normalconjlemma},
\[ (g \varepsilon_2 g^{-1})^{-1}\varepsilon_1 y = y (g \varepsilon_2 g^{-1})^{-1}\varepsilon_1, \]
for all $y\in G$ and so $(g \varepsilon_2 g^{-1})^{-1}\varepsilon_1 \in Z(G) = \{\id \}$.  Thus we have that $\varepsilon_1^{-1} = (g \varepsilon_2 g^{-1})^{-1}$, so $\varepsilon_1 = g \varepsilon_2 g^{-1}$.  So now notice $S = g T g^{-1}$ is a maximal torus containing $\varepsilon_1$.  The group $Z_G(\varepsilon_1)$ contains $S$ and $T$, so there exists an $x\in Z_G(\varepsilon_1)$ such that $x S x^{-1} = T$.  We know that $S = gTg^{-1}$ so
\[ xg T g^{-1}x^{-1} = xg T (xg)^{-1} = T, \]
which has $xg \in N_G(T)$.  We notice that 
\begin{align*}
\I_{xg}\I_{\varepsilon_2} \I_{xg}^{-1} &= \I_{ xg \varepsilon_2 (xg)^{-1} } \\
&= \I_{ xg \varepsilon_2 g^{-1}x^{-1} } \\
&= \I_{x\varepsilon_1 x^{-1}} \\
&= \I_{\varepsilon_1},
\end{align*}
which from the previous argument we have $(xg)\varepsilon_2(xg)^{-1} = \varepsilon_1$.

\end{proof}

Using the previous proposition it is possible to find elements $n,m \in N_G(T)$ such that 
\[ t_{(-1,-1)} = n \left(t_{(-1,1)}\right) n^{-1} = m \left(t_{(1,-1)}\right) m^{-1}. \]

It is also possible to show, and perhaps more illustrative, that they leave isomorphic quaternion subalgebras invariant, and thus by Corollary \ref{conjint2} provide us with isomorphic $k$-involutions.

\newtheorem{tconj}[subsubsection]{Proposition}
\begin{tconj}
\label{tconj}
$t_{(-1,-1)} \cong t_{(-1,1)} \cong t_{(1,-1)}.$
\end{tconj}
\begin{proof}
Let $G=\Aut(C)\supset T = \left\{ \diag(1,\beta \gamma, \beta^{-1} \gamma^{-1}, 1, \gamma^{-1}, \beta, \beta^{-1}, \gamma ) \ \big| \ \beta, \gamma \in k^* \right\}$, then $G$ is an algebraic group over the field $k$, $\ch(k)\neq 2$.  The automorphism $t_{(-1,-1)}$ leaves the split quaternion subalgebra $(M_2(k),0)$ elementwise fixed. \

 The element of order $2$,
\[ t_{(1,-1)} = \diag(1,-1,-1,1,-1,1,1,-1), \]
leaves the quaternion subalgebra,
\[
k 
\underbrace{
\left(
\begin{bmatrix}
1 & \\
 & 1 \\
\end{bmatrix},
0
\right)
}_{e}
\bigoplus
k
\underbrace{
\left( 0,
\begin{bmatrix}
 & 1 \\
-1 &  \\
\end{bmatrix}
\right)
}_{a}
\bigoplus
k
\underbrace{
\left( 0,
\begin{bmatrix}
 & 1 \\
 1 & \\
\end{bmatrix}
\right)
}_{b}
\bigoplus
k
\underbrace{
 \left(
\begin{bmatrix}
1 & \\
 & -1 \\
\end{bmatrix},
0 \right)
}_{ab}
,
\]
elementwise fixed.  Notice that $(b-a)(e+ab)=(0,0)$, and so the quaternion subalgebra is split.

And 
\[ t_{(-1,1)} = \diag(1,-1,-1,1,1,-1,-1,1), \]
leaves the quaternion algebra
\[
k 
\underbrace{
\left(
\begin{bmatrix}
1 & \\
 & 1 \\
\end{bmatrix},
0
\right)
}_{e}
\bigoplus
k 
\underbrace{
\left(
\begin{bmatrix}
1 & \\
 & -1 \\
\end{bmatrix},
0 \right)
}_{a}
\bigoplus
k
\underbrace{
\left( 0,
\begin{bmatrix}
1 & \\
 & 1 \\
\end{bmatrix}
\right)
}_{b}
\bigoplus
k
\underbrace{
\left( 0,
\begin{bmatrix}
1 & \\
 & -1 \\
\end{bmatrix}
\right)
}_{ab},
\]
elementwise fixed.   Notice that $(ab + b)(e + a) = (0,0)$, and so the quaternion subalgebra is split.  Since over a given field $k$ every split quaternion subalgebra is isomorphic, we have that $t_{(-1,-1)} \cong t_{(-1,1)} \cong t_{(1,-1)}$.
\end{proof}

\newtheorem{cortconj}[subsubsection]{Corollary}
\begin{cortconj}
\label{cortconj}
$\mathcal{I}_{t_{(-1,-1)}} \cong \mathcal{I}_{t_{(-1,1)}} \cong \mathcal{I}_{t_{(1,-1)}}.$
\end{cortconj}

So from now on we refer to a representative of the congruence class containing $\mathcal{I}_{t_{(-1,-1)}}$,  $\mathcal{I}_{t_{(-1,1)}}$, and  $\mathcal{I}_{t_{(1,-1)}}$ as $\mathcal{I}_t$, when there is no ambiguity.

\newtheorem{splitinv}[subsubsection]{Lemma}
\begin{splitinv}
There is only one isomorphy class of $k$-involutions when $k$ is a finite field of order greater than $2$, complex numbers, $p$-adic fields when $p>2$, or when $k$ is a complete, totally imaginary algebraic number field.
\end{splitinv}
\begin{proof}
In these cases only split quaternion algebras exist, \cite{O'M63}, \cite{SV00}.
\end{proof}

In \cite{Yo90} Yokota talks about the maps $\gamma, \gamma_C$, and $\gamma_H$ and shows are they isomorphic, and that they are also isomorphic to any composition of maps between them.  In his paper he defines a conjugation coming from complexification.  In particular we can look at $\gamma_H$, which is the complexification conjugation on the quaternion level of an octonion algebra over $\mathbb{R}$.  If we take $u + vc \in H \oplus Hc$ where $u,v\in H$ and $c\in H^{\perp}$ his map is
	\[ \gamma_H(u+vc) = u-vc, \]
which in our presentation of the octonion algebra would look like, 
	\[ \gamma_H \left(
		\begin{bmatrix}
		u_{11} & u_{12} \\
		u_{21} & u_{22}
		\end{bmatrix},
		\begin{bmatrix}
		v_{11} & v_{12} \\
		v_{21} & v_{22}
		\end{bmatrix} \right)
		= \left(
		\begin{bmatrix}
		u_{11} & -u_{12} \\
		-u_{21} & u_{22}
		\end{bmatrix},
		\begin{bmatrix}
		v_{11} & -v_{12} \\
		-v_{21} & v_{22} 
		\end{bmatrix} \right),
		\]
and corresponds to our map $\mathcal{I}_{t_{(-1,1)}}$.

\subsection{Maximal $\theta$-split torus}

Rather than trying to find a maximal $\theta$-split torus, where $\theta \cong \mathcal{I}_t$, and then computing its maximal $k$-split subtorus, we find a $k$-involution $\theta$ that splits our already maximal $k$-split torus of the form
\[  T = \left\{ \diag(1,\beta \gamma, \beta^{-1} \gamma^{-1}, 1, \gamma^{-1}, \beta, \beta^{-1}, \gamma ) \ \big| \beta, \gamma \in k^* \right\}. \]
It is straight forward to check that
\[
s=
\begin{bmatrix}
& & & 1 \\
& & 1 & \\
& 1 & & \\
1 & & &
\end{bmatrix}
\bigoplus
\begin{bmatrix}
& & & 1 \\
& & 1 & \\
& 1 & & \\
1 & & &
\end{bmatrix}
, \]
is an element of $\Aut(C)$, where $C$ is the split octonion algebra described above over a field $k$, $\ch(k) \neq 2$.  It is immediate that $T$ is a $\mathcal{I}_s$-split torus.

\newtheorem{maxsplit}[subsubsection]{Proposition}
\begin{maxsplit}
$T$ is a maximal $(\mathcal{I}_s,k)$-split torus.
\end{maxsplit}
\begin{proof}
Notice first that if $t \in T$ that $\mathcal{I}_s(t)=t^{-1}$, and next that $T$ is $k$-split and is a maximal torus.
\end{proof}

\newtheorem{scongt}[subsubsection]{Proposition}
\begin{scongt}
\label{scongt}
$\mathcal{I}_s \cong \mathcal{I}_t$
\end{scongt}
\begin{proof}
The element $s$ is an automorphism of order $2$ of $C$, our split octonion algebra described above, that leaves the following quaternion algebra fixed elementwise,
\[
k 
\underbrace{
\left(
\begin{bmatrix}
1 & \\
 & 1 \\
\end{bmatrix},
0
\right)
}_{e}
\bigoplus
k 
\underbrace{
\left(
\begin{bmatrix}
 & 1 \\
1 &  \\
\end{bmatrix},
0 \right)
}_{a}
\bigoplus
k
\underbrace{
\left( 0,
\begin{bmatrix}
1 & \\
 & 1 \\
\end{bmatrix}
\right)
}_{b}
\bigoplus
k
\underbrace{
\left( 0,
\begin{bmatrix}
 & 1 \\
1 &  \\
\end{bmatrix}
\right)
}_{ab}.
\]
Notice that $(b+ab)(e+a+b+ab)=0$, and so the quaternion subalgebra is split.
\end{proof}

\subsection{Another isomorphy class of $k$-involutions over certain fields}

We have seen that our maximal torus $T=T_{\mathcal{I}_s}^-$, and so we can look at elements of $T$ for $k$-inner elements of $\mathcal{I}_s$ that will give us new conjugacy classes over fields for which quaternion division algebras can exist.  The fields we are interested in include the real numbers, 2-adics, and rationals.

\newtheorem{nonsplitq}[subsubsection]{Lemma}
\begin{nonsplitq}
\label{nonsplitq}
For $C$ a split octonion algebra over a field $k=\mathbb{R}, \mathbb{Q}_2, \mathbb{Q}$,
\[ s\cdot t_{(1,-1)} \in \Aut(C), \]
 leaves a quaternion division subalgebra elementwise fixed.
\end{nonsplitq}
\begin{proof}
The element $s\cdot t_{(1,-1)} \in \Aut(C)$ leaves the following quaternion subalgebra elementwise fixed,
\[
k 
\underbrace{
\left(
\begin{bmatrix}
1 & \\
 & 1 \\
\end{bmatrix},
0
\right)
}_{e}
\bigoplus
k
\underbrace{
\left( 0,
\begin{bmatrix}
1 & \\
 & -1 \\
\end{bmatrix}
\right)
}_{a}
\bigoplus
k
\underbrace{
\left( 0,
\begin{bmatrix}
 & 1 \\
1 &  \\
\end{bmatrix}
\right)
}_{b}
\bigoplus
k 
\underbrace{
\left(
\begin{bmatrix}
 & 1 \\
-1 &  \\
\end{bmatrix},
0 \right)
}_{ab}.
\]
All basis elements are such that $x \bar{x} = 1$, and so have a norm isomorphic to the 2-Pfister form $\left( \frac{-1,-1}{k} \right)$, where $k=\mathbb{R}, \mathbb{Q}_2, \mathbb{Q}$, which corresponds to a quaternion division algebra over each respective field.  Moreover, over $k=\mathbb{R}$ or $\mathbb{Q}_2$ there is only one quaternion division algebra up to isomorphism.
\end{proof}

\newtheorem{nonsplitqQp}[subsubsection]{Lemma}
\begin{nonsplitqQp}
\label{nonsplitqQp}
For $C$ a split octonion algebra over a field $k=\mathbb{Q}_p$ and $p>2$ $s \cdot t_{(-N_p,-pN_p^{-1})}$ leaves a divison quaternion algebra elementwise fixed.
\end{nonsplitqQp}
\begin{proof}
The element $s \cdot t_{(-N_p,-pN_p^{-1})}$ leaves the following quaternion subalgebra elementwise fixed,
\[
k 
\underbrace{
\left(
\begin{bmatrix}
1 & \\
 & 1 \\
\end{bmatrix},
0
\right)
}_{e}
\bigoplus
k
\underbrace{
\left( 0,
\begin{bmatrix}
 & -N_p \\
 1 &  \\
\end{bmatrix}
\right)
}_{a}
\bigoplus
k
\underbrace{
\left( \begin{bmatrix}
 & p \\
1 &  \\
\end{bmatrix},0
\right)
}_{b}
\bigoplus
k 
\underbrace{
\left(0,
\begin{bmatrix}
N_p &  \\
 & -p \\
\end{bmatrix}
 \right)
}_{ab},
\]
with $\mathbb{Q}_p^*/ (\mathbb{Q}_p^*)^2 = \{ 1, p, N_p, pN_p\}$.  This algebra is ismorphic to $\left( \frac{p,N_p}{\mathbb{Q}_p} \right)$, which is a representative of the unique ismorphy class of quaternion division algebras for a given $p$.
\end{proof}

\newtheorem{maintheorem}[subsubsection]{Theorem}
\begin{maintheorem}
\label{maintheorem}
Let $\theta = \mathcal{I}_s$ and $G=\Aut(C)$ where $C$ is a split octonion algebra over a field $k$, then
\begin{enumerate}  
\item when $k=\mathbb{R}, \mathbb{Q}, \mathbb{Q}_2$;  $\theta$ and $\theta \circ \mathcal{I}_{ t_{(1,-1)} }$ are representatives of $2$ isomorphy classes of $k$-involutions of $G$.  In the cases $k=\mathbb{R}$ or $\mathbb{Q}_2$ these are the only cases, but this is not true for $k=\mathbb{Q}$.
\item when $k=\mathbb{Q}_p$ and $p>2$, we have two isomorphy classes of $k$-involutions.
\item when $k=\mathbb{C}$ and $\mathbb{F}_p$ with $p>2$; there is only one isomorphy class of $k$-involutions of $G$.
\end{enumerate}
\end{maintheorem}
\begin{proof}

For part (1) we only need to notice that over the fields $\mathbb{R}$, $\mathbb{Q}$ and $\mathbb{Q}_2$ that $\theta$ and $\theta \circ \mathcal{I}_{t_{(1,-1)}}$ leave nonisomorphic subalgebras elementwise fixed and so by Theorem \ref{jake}, Corollary \ref{cortconj}, and Proposition \ref{scongt} they are not isomorphic.  And by Theorem \ref{splitcomp} and Lemma \ref{nonsplitq} these are the only $2$ quaternion subalgebras up to isomorphism.  There are no other possible isomorphy classes.

For part (2) by Theorem \ref{splitcomp}, Lemma \ref{nonsplitqQp}.

For part (3) by Theorem \ref{splitcomp}, Proposition \ref{tconj}, and Corollary \ref{cortconj} we have the result. 
\end{proof}

\subsection{Fixed point groups}

In order to compute the fixed point groups of each $k$-involution we first look at how such elements of $\Aut(C)$ act on $C$.

\newtheorem{fixedlemma}[subsubsection]{Lemma}
\begin{fixedlemma}
\label{fixedlemma}
Let $t\in \Aut(C)=G$ such that $t^2 = \id$ and $D \subset C$ the quaternion algebra elementwise fixed by $t$ then $f \in  G^t := G^{\I_t} = \{g \in G \ | \ \I_t(g)=g \}$ if and only if $f$ leaves $D$ invariant.
\end{fixedlemma}
\begin{proof}
($\Rightarrow$):  Let $D \subset C$ be fixed elementwise by $t$ and $f\in G^t$, then $\I_t(f)=f$.  Now let $c\in C$ be any element of the octonion algebra containing $D$, then we can write $C = D \oplus D^{\perp}$ and $c=a+b$ where $a \in D$ and $b \in D^{\perp}$.  Since $t$ is a $k$-involution it has only $\pm 1$ as eigenvalues and $t(a+b) = a - b$.  Furthermore, 
	\begin{align*}
	 \I_t(f)(c) &= tft^{-1}(c) \\
		&=tft(c) \\
		&=tft(a+b) \\
		&=tft(a)+tft(b) \\
		&=tf(a)+f(b),
	\end{align*}
and since $\I_t(f)(c) = f(c)=f(a+b)$ for all $c\in C$,
	\[ f(a)+f(b) = tf(a) + f(b). \]
From this we can conclude that $tf(a) = f(a)$ so $f(a) \in D$. \

($\Leftarrow$):  If we assume, conversely, that $D \subset C$ is the subalgebra fixed elementwise by $t$ and $f(D)=D$ then $f(D^{\perp}) = D^{\perp}$, and
	\begin{align*}
	tft^{-1}(c) &= tft(a+b) \\
		&= tft(a)+tft(b) \\
		&= tf(a) - tf(b) \\
		&= f(a) + f(b) \\
		&= f(c),
	\end{align*}
for all $c\in C$ and we have the result.
\end{proof}

Every involution in $\Aut(C)$ leaves some quaternion algebra, $D$, elementwise fixed.  Now we need only to see what automorphisms of $C$ leave $D$ invariant.  In every case we will consider if we have a fixed quaternion algebra $D \subset C$, then $C= D\oplus D^{\perp}$ with respect to $N$.  The automorphisms of $C$ that leave $D$ invariant, denoted $\Aut(C,D)$ are of the form,
	\[ s(x+ya) = s(x) +s(y)s(a), \]
where $s \in \Aut(C,D)$; $x,y \in D$ and $a\in D^{\perp}$ such that $N(a) \neq 0$.  Since $s$ leaves $D$ invariant we can further see that
	\[ s(y) \in D, \text{ and } s(a) \in D^{\perp} \text{ such that } N(s(a)) \neq 0, \]
and we can write
	\[ s(x+ya) = s_{dp}(x+ya) = dxd^{-1} + (pdyd^{-1}) a, \]
where $d, p \in D$ and $N(d)\neq 0$, $N(p)=1$.  For more details see \cite{SV00}. 

First let us consider the case where $D$ is a split quaternion algebra.  In this case $D \cong M_2(k)$ and $d\in \GL _2(k)$ and $p \in \SL_2(k)$.  

\newtheorem{fixsplit}[subsubsection]{Proposition}
\begin{fixsplit}
When $t \in \Aut(C)$ is an involution and leaves $D \subset C$, a split quaternion subalgebra, elementwise fixed, then 
	\[ G^t := G^{\I_t} \cong  \PGL_2(k) \times \SL_2(k). \]
\end{fixsplit}	
	\begin{proof}
	If we consider the map 
		\[ \psi: \GL _2(k) \times \SL_2(k) \to \Aut(C,D) \text{, where } (d,p) \mapsto s_{dp}, \]
	is surjective.  The kernel is given by $\ker(\psi) = \{ (\alpha \cdot e , e ) \ | \ \alpha \in k^*\}$.
	\end{proof}

In the case where the involution $t$ leaves a quaternion division algebra $D$ invariant, we have the same initial set up, i.e., $s\in \Aut(C,D)$ then
	\[ s(x+ya) = s_{dp}(x+ya) = d x d^{-1} + (pdyd^{-1}) a, \]
where $N(d) \neq 0$ and $N(p)=1$, only $D \not\cong M_2(k)$ it is isomorphic to Hamilton's quaternions over $k$.  In this case $N(d) \neq 0$ only tells us that $d\neq 0 \in D$.  

\newtheorem{fixdiv}[subsubsection]{Proposition}
\begin{fixdiv}
When $t \in \Aut(C)$ is an involution and leaves $D \subset C$, a quaternion division algebra, elementwise fixed, then $G^t \cong \SO(D_0,N) \times \Sp(1)$ where $D_0 = ke^{\perp}$.
\end{fixdiv}
	\begin{proof}
	Then $N(p) = p \bar{p} = 1$ tells us that $p \in \Sp(1)$ is the group of $1 \times 1$ symplectic matrices over $D$.  If we consider the surjective homomorphism 
		\[\psi: D^* \times \Sp(1) \to \Aut(C,D) \text{ where } (d,p) \mapsto s_{dp},\]
	and $D^*=D-\{0\}$ is the group consisting of the elements of $D$ having inverses, its kernel, $\ker(\psi) = \{ (\alpha \cdot e,e) \ | \  \alpha \in k^* \}$.  So we have 
		\[ D^*/Z(D^*) \times \Sp(1) \cong \Inn(D^*) \times \Sp(1). \]
	Jacobson tells us, \cite{Ja58}, that all automorphisms of $D^*$ are inner and leave the identity fixed.  Also included in \cite{Ja58} is that $\Inn(D^*) \cong \SO(D_0,N)$, where $D_0 \subset D$ such that $D_0 = e^{\perp}$, the three dimensional subspace, and $\SO(D_0,N)$ is the group of rotations of $D_0$ with respect to $N|_{D_0}$.
	\end{proof}

\section{Concluding remarks}

\subsection{$k=\mathbb{Q}$}

Isomorphy classes of quaternion algebras over $k$ are given by equivalence classes of $2$-Pfister forms $\left(\frac{\alpha,\beta}{k}\right)$ over $k$ corresponding to the quadratic form,
	\[ N(x) = x_0^2 + (-\alpha)x_1^2 + (-\beta)x_2^2 + (\alpha\beta)x_3^2, \]
while octonion algebras depend on $3$-Pfister forms, $\left( \frac{\alpha,\beta,\gamma}{k} \right)$ with quadratic form, 
	\[ N(x) = x_0^2 + (-\alpha)x_1^2 +(-\beta)x_2^2 + (\alpha\beta)x_3^2 + (-\gamma)x_4^2 + (\alpha\gamma)x_5^2 + (\beta\gamma)x_6^2 + (-\alpha\beta\gamma)x_7^2. \]

It is not difficult to show that for a prime $p$, such that $p \equiv 3 \mod 4$, 
	\[ \left( \frac{ -1, p }{\mathbb{Q}} \right) \]
is a division algebra.  Further, for $p$ and $q$ distinct primes both equivalent to $3 \mod 4$
	\[  \left( \frac{ -1, p }{\mathbb{Q}} \right) \not\cong  \left( \frac{ -1, q }{\mathbb{Q}} \right), \]
see \cite{Pi82} Exercise 1.7.4.

\newtheorem{rationalex}[subsubsection]{Example}
\begin{rationalex}
Let $C = \left(M_2(\mathbb{Q}),M_2(\mathbb{Q}) \right)$.  We can find an involution of $\Aut(C)$ leaving $D \cong \left( \frac{ -1, p }{\mathbb{Q}} \right)$ elementwise fixed by first constructing a basis for $D$.  If we pick
	\[ a = \left(  0,
		\begin{bmatrix}
		 & 1 \\
		1 & 
		\end{bmatrix}
		\right)
	\text{ and }
		b = \left(
		\begin{bmatrix}
		 & p \\
		1 & 
		\end{bmatrix}, 
		0 \right), \]
then $a^2 = -1$ and $b^2 = p$.  Then the other basis elements of $D$ are $e= (\id,0)$ and
	\[ ab = \left( 0,
		\begin{bmatrix}
		-1 & \\
		 & -p
		\end{bmatrix}
		\right). \]
Let $p \equiv 3 \mod 4$, then $D$ is a division algebra and
	\[
		s_p = \begin{bmatrix}
				& & & 1 \\
				& & p & \\
				& p^{-1} & & \\
				1 & & &
			\end{bmatrix}
			\bigoplus
			\begin{bmatrix}
				& & & p^{-1} \\
				& & 1 & \\
				& 1 & & \\
				p & & &
			\end{bmatrix} \in \Aut(C), \]
leaves $D$ elementwise fixed.  Notice $s_p = s \cdot t_{(p,1)}$.  There is a $\mathbb{Q}$-involution for each $p \equiv 3 \mod 4$ of which there are an infinite number.
\end{rationalex}

Over $\mathbb{Q}$ there are only $2$ octonion algebras, $\left( M_2(\mathbb{Q}), M_2(\mathbb{Q}) \right)$ and a division algebra.  When we consider quaternion algebras over $\mathbb{Q}$ we get one split quaternion algebra, but we get exactly one division algebra for each unique set of an even number of real or finite places of $\mathbb{Q}$, \cite{O'M63} and  \cite{SV00}.  \

\subsection{$k$ is an algebraic number fields}

When the quaternion or octonion algebra is taken over an algebraic number field we have that $2$ quadratic forms are equivalent if and only if they are equivalent over all local fields $k_{\nu}$ where $\nu$ varies over all places of $k$, which is a result of Hasse's Theorem \cite{Se73}.  Hasse's Theorem also tells us the number of possible anisotropic quadratic forms corresponding to nonisomorphic octonion algebras is given an upper bound of $2^r$ where $r$ is the number of real places of $k$ \cite{SV00}. \

\newpage

\bibliographystyle{plain}

\end{document}